\documentclass[11pt]{article}
\usepackage{graphicx}
 
\usepackage{amsthm, amsmath, amssymb, tikz, bm}
\usepackage{mathtools}
\usepackage{xifthen}
\usepackage{comment}

\usepackage{thm-restate}
\theoremstyle{plain}
\usepackage[margin=1.1in]{geometry} 

\usepackage[shortlabels]{enumitem}
\usepackage{todonotes}
\usepackage[colorlinks=true, linkcolor=blue,citecolor=blue,
urlcolor=blue]{hyperref}
\usepackage[noabbrev,capitalise]{cleveref}
\usetikzlibrary{calc,shapes, backgrounds}
\usetikzlibrary[patterns]
\usetikzlibrary{arrows.meta}
\usetikzlibrary{decorations.markings}

\newtheorem{theorem}{Theorem}[section]
\newtheorem{lemma}[theorem]{Lemma}
\newtheorem{corollary}[theorem]{Corollary}

\newtheorem{observation}[theorem]{Observation}

\newcommand{\dss}{\displaystyle\sum}

\newcommand{\lp}{\left (}
\newcommand{\rp}{\right )}

\DeclarePairedDelimiter{\ceil}{\lceil}{\rceil}
\DeclarePairedDelimiter{\floor}{\lfloor}{\rfloor}

\makeatletter
               {\list{}{\leftmargin=0pt 
                        \labelwidth\z@ \itemindent-\leftmargin
                        }}%
               {\endlist}
\makeatother

\title{On tight $(k,\ell)$-stable graphs}

\author{
Xiaonan Liu \thanks{Department of Mathematics, Vanderbilt University, Nashville, TN, 37240
({\tt xiaonan.liu@vanderbilt.edu}).}
\and
Zi-Xia Song \thanks{Department of Mathematics, University of Central Florida, Orlando, FL, 32816
({\tt Zixia.Song@ucf.edu}).   Supported by  NSF grant    DMS-2153945.}
\and
Zhiyu Wang \thanks{Department of Mathematics, Louisiana State University, Baton Rouge, LA, 70803
({\tt zhiyuw@lsu.edu}).}
}
\begin{document}

\maketitle

\begin{abstract}
For integers $k>\ell\ge0$, a graph $G$ is $(k,\ell)$-stable if  $\alpha(G-S)\geq \alpha(G)-\ell$ for every    
$S\subseteq V(G)$ with $|S|=k$. A recent result of Dong and Wu [SIAM J.
Discrete Math., 36 (2022) 229--240] shows that every $(k,\ell)$-stable 
graph $G$  satisfies $\alpha(G) \le  \lfloor ({|V(G)|-k+1})/{2}\rfloor+\ell$.  A $(k,\ell)$-stable graph $G$   is   tight if $\alpha(G) = \lfloor ({|V(G)|-k+1})/{2}\rfloor+\ell$; and  $q$-tight for some integer $q\ge0$ if $\alpha(G) = \lfloor ({|V(G)|-k+1})/{2}\rfloor+\ell-q$.
In this paper, we first prove  that for all $k\geq 24$, the only tight $(k, 0)$-stable graphs are $K_{k+1}$ and  $K_{k+2}$, answering a question of Dong and Luo [arXiv: 2401.16639]. We then prove that  for all nonnegative integers $k, \ell, q$ with $k\geq 3\ell+3$, every $q$-tight $(k,\ell)$-stable graph has at most  $k-3\ell-3+2^{3(\ell+2q+4)^2}$ vertices, answering a question of Dong and Luo in the negative. 
\end{abstract}

\section{Introduction}
 All graphs   considered in this paper  are finite, simple, and undirected.   Give a graph $G$,    we  use $V(G)$ to denote the vertex set, $E(G)$ the edge set,    $\delta(G)$ the minimum degree, $\Delta(G)$ the maximum degree,  $\alpha(G)$ the independence number,  and $\omega(G)$ the clique number of $G$.
The resilience of graph properties  is a fundamental question in graph theory. 
Motivated by studies on the Erd\H{o}s-Rogers function, Dong and Wu~\cite{DW22} investigated the resilience of graph independence number with respect to removing vertices. 
Given a graph $G=(V,E)$ and a subset $S$ of $V(G)$, let $G-S$ be the induced subgraph of $G$ with vertex set  $V(G)- S$. For every $k> \ell\geq0$, we say that a graph $G$ is \textit{$(k,\ell)$-stable} if for every $S\subseteq V(G)$ with $|S|=k$, $\alpha(G-S)\geq \alpha(G)-\ell$.  \medskip

 \cref{thm:independence_num_Dong_Wu} is a result of Dong and  Wu~\cite{DW22} on   the independence number of $(k,\ell)$-stable graphs; it  was used to  obtain new values of the Erd\H{o}s-Rogers function (see \cite[Theorem 2]{DW22}). 

\begin{theorem}[\cite{DW22}]\label{thm:independence_num_Dong_Wu}
For all integers $k>\ell \geq 0$, every  $(k,\ell)$-stable graph on $n$ vertices satisfies 
\[\alpha(G)\leq \floor*{\frac{n-k+1}{2}}+\ell.\]
 
\end{theorem}

A $(k,\ell)$-stable graph $G$ is   \textit{tight} if $ \alpha(G)= \floor*{\frac{n-k+1}{2}}+\ell$. Dong and Wu~\cite{DW22} proved that  there exists a tight $n$-vertex $(k,\ell)$-stable graph for every $n$, provided that  $k=\ell+1$ or  $k=\ell+2$. Very recently, Dong and  Luo~\cite{DL24} studied the structure of tight $(k,0)$-stable graphs. In particular, they proved that while there exist tight $(1,0)$-stable and $(2,0)$-stable $n$-vertex graphs for every $n$, a tight $(k,0)$-stable graph has at most $k+6$ vertices for all $k\ge3$.

\begin{theorem}[\cite{DL24}]\label{thm:(k,0)-tight-vertices}
   For all $k \geq 3$, every tight $(k, 0)$-stable graph has at most $k+6$ vertices. 
\end{theorem}

Observe that for every $k\geq 1$, a tight $(k,0)$-stable graph always exists, as demonstrated by the complete graphs  $K_{k+1}$ and  $K_{k+2}$. However, for $k\geq 3$, Dong and Luo \cite{DL24} asked whether there exists any other natural infinite family $G_k$ that is tight $(k,0)$-stable. In particular, they asked the following: does there exist a positive integer $k_0$ such that for all $k\geq k_0$, the only tight $(k,0)$-stable graphs are  $K_{k+1}$  (and  $K_{k+2}$)?
We first answer this  question in the positive. Let $R(s,t)$ denote the Ramsey number of $K_s$ and $K_t$, i.e., the minimum integer $N$ such that any graph on $n\geq N$ vertices contains either  an independent set   of $s$ vertices or a clique   of  $t$ vertices. Recall that $R(4,5)=25$. We  prove  Theorem~\ref{thm:tight (k,0)-unique} here as its proof is very short. Our proof is inspired by the argument of Dong and Wu~\cite{DW22} who proved that there is no tight $(3,0)$-stable graph on six vertices. 

\begin{theorem}\label{thm:tight (k,0)-unique}
For all $k\geq 24$, the only tight $(k,0)$-stable graphs are $K_{k+1}$ and $K_{k+2}$.
\end{theorem}

\begin{proof}
Let  $G$ be a tight $(k,0)$-stable graph on $n\ge k+1$ vertices, where  $k\geq  24$. Then 
\[\alpha(G) = \floor*{\frac{n-k+1}{2}}.\]
It suffices to show that $\alpha(G)=1$. Suppose $\alpha(G)\ge2$.
By Theorem \ref{thm:(k,0)-tight-vertices},   $n= k+c$   for some  positive  integer $c\leq 6$. 
Moreover, since $G$ is   $(k,0)$-stable, for every  induced subgraph  $T$ of $G$ with $|V(T)|=c$, 
\[ \alpha(T) = \alpha(G) = \floor*{\frac{k+c-k+1}{2}} = \floor*{\frac{c+1}{2}}.\]

 It follows that  $3\leq c\leq 6$ because $\alpha(G)\ge2$; moreover, $\alpha(G)=2$ when $3\le c\le4$ and $\alpha(G)=3$ when $5\le c\le6$. 
This, together with the fact that   $n \geq k+1\geq 25=R(4,5)$,  implies that   $\omega(G)\ge5$.    Let $K$ be a clique of $G$ with $|K|=5$ and $T$ be an induced subgraph of $G$ with $|V(T)|=c$ such that $|V(T)\cap K|$ is maximum. Then $\alpha(T)=1$ if  $c\leq 5$ and $\alpha(T)\le 2$ if $c=6$,    contrary to the fact that $\alpha(G)=\alpha(T)$.  This proves that for all $k\geq 24$, the only tight $(k,0)$-stable graphs are $K_{k+1}$ and $K_{k+2}$, as desired.
\end{proof}

In another direction, for every $\ell\geq 0$, Dong and Wu \cite{DW22} constructed a sequence of $n$-vertex $(\ell+3,\ell)$-stable graphs with independence number $n/2-O(\sqrt{n})$.  Alon~\cite{Alo21} extended the above result to a more general setting by showing that for every $k>\ell\geq 0$, there exists a sequence of $n$-vertex $(k,\ell)$-stable graphs with independence number $n/2-o(n)$. For $k=3$, Wu and Luo~\cite{DL24} asked whether there exists some constant $c>0$ such that there is a sequence of $(3,0)$-stable graphs $G$, with vertex number $n\to \infty$ and $\alpha(G)\geq |V(G)|/2-c$? Here we answer their question in the negative.

\begin{theorem}\label{thm:Question2-disprove}
    For every constant $q>0$ and $k\geq 3$, there exists some constant $C_{q,k}$ such that every $(k,0)$-stable graph $G$ with $\alpha(G)\geq |V(G)|/2-q$ has at most $C_{q,k}$ vertices.
\end{theorem}

We prove the following  stronger result. 
For integers $k>\ell\geq 0$ and $q\ge0$, we call a  graph $G$ \textit{$q$-tight} $(k,\ell)$-stable if $G$ is $(k,\ell)$-stable and  $\alpha(G) = \floor*{\frac{n-k+1}{2}}+\ell-q$. When $q=0$, a $q$-tight $(k,\ell)$-stable graph is exactly a tight $(k,\ell)$-stable graph. 

\begin{restatable}{theorem}{qtight}\label{thm:q-tight(k,l)}
  For nonnegative integers $k, \ell, q$ with $k\geq 3\ell+3 $, every $q$-tight $(k,\ell)$-stable graph has at most  $(k-3\ell-3)+2^{3(\ell+2q+4)^2}$ vertices.
\end{restatable}
We obtain a slightly better upper bound  for  $q$-tight $(k,0)$-stable graphs.

\begin{theorem}\label{thm:q-tight(k,0)-stable}
  For integers $k\geq 3$ and $q\geq 0$,
every $q$-tight $(k,0)$-stable graph has at most $(k-3)+2^{3(2q+3)^2}$ vertices.
\end{theorem}

It is worth noting that Theorem \ref{thm:Question2-disprove} follows from Theorem \ref{thm:q-tight(k,0)-stable} by setting $C_{q,k}=k-3+2^{3(2q+3)^2}$.  Recall that $0$-tight $(k,\ell)$-stable graphs are tight $(k,\ell)$-stable graphs. Combining this with Theorem~\ref{thm:q-tight(k,l)} leads to Corollary~\ref{thm:(k,1)-stable-order}.


\begin{corollary}\label{thm:(k,1)-stable-order}
    For nonnegative integers $k, \ell$ with $k\geq 3\ell+3$, every  tight $(k,\ell)$-stable graph has at most  $(k-3\ell-3)+2^{3(\ell+4)^2}$ vertices. In particular, for all $k\geq 6$, every tight $(k,1)$-stable graph has at most $k-6+2^{75}$ vertices.
\end{corollary}

One can now ask the same question: does there exist a positive integer $k_0$ such that for all $k\geq k_0$, the only tight $(k,1)$-stable graphs are $K_k$,  $K_{k+1}$ and  $K_{k+2}$? The answer is negative in this case, as such a graph may have independence number two. However, the number of vertices of such graphs is at most $k+2$.  The proof of Theorem~\ref{thm:tight (k,1)-unique} is similar  to the proof of Theorem~\ref{thm:tight (k,0)-unique}. We provide a   proof here for completeness. 

\begin{theorem}\label{thm:tight (k,1)-unique}
For all $k\geq R(2^{74}, 2^{75}-6)$, the only tight $(k,1)$-stable graphs are $K_k$, and  spanning subgraphs of $K_{k+1}$ or  $K_{k+2}$ with independence number two; in particular, for all $k\geq R(2^{74}, 2^{75}-6)$, every tight $(k,1)$-stable graph has at most $k+2$ vertices. 
\end{theorem}

\begin{proof} Let  $G$ be a tight $(k,1)$-stable graph on $n\ge k$ vertices, where  $k\geq  R(2^{74}, 2^{75}-6)$. Then 
\[\alpha(G) = \floor*{\frac{n-k+1}{2}}+1.\tag{a}\]
It suffices to show that $\alpha(G)\le2$. Suppose $\alpha(G)\ge3$.
By Corollary \ref{thm:(k,1)-stable-order},   $n= k+c$   for some  positive  integer $c\leq 2^{75}-6$.  It follows that 
\[\alpha(G) = \floor*{\frac{k+c-k+1}{2}}+1 = \floor*{\frac{c+1}{2}}+1\le  2^{74}-1.\]
Since $G$ is   $(k,1)$-stable, for every  induced subgraph  $T$ of $G$ with $|V(T)|=c$, 
 \[ \alpha(T) \ge  \alpha(G)-1 = \floor*{\frac{c+1}{2}}.\]
Note that    $n \geq k\geq R(2^{74}, 2^{75}-6)$ and $\alpha(G)\le 2^{74}-1$;   we see that   $\omega(G)\ge2^{75}-6$.    Let $K$ be a clique of $G$ with $|K|=2^{75}-6$ and $T$ be an induced subgraph of $G$ with $|V(T)|=c$ such that $|V(T)\cap K|$ is maximum. Then $\alpha(T)=1$,  contrary to the fact that $\alpha(T)\ge\alpha(G)-1\ge2$.  This proves that  $\alpha(G)\le2$,  and so $n\le k+2$ by $(a)$. It follows that $G$ is $K_k$ if $\alpha(G)=1$ and  $G$ is a spanning subgraph of $K_{k+1}$ or $K_{k+2}$ when  $\alpha(G)=2$. 
\end{proof}

\noindent{\bf Remark.} Following Theorem \ref{thm:tight (k,1)-unique}, one can show that for a fixed integer $\ell\geq 0$ and every $k\geq R(2^{3(\ell+4)^2-1}, 2^{3(\ell+4)^2}-3\ell-3)$, every tight $(k,\ell)$-stable graph has at most $k+2$ vertices.

It would be interesting to know whether  Theorem \ref{thm:q-tight(k,l)} holds when $k\le  3\ell+2$. For $2\ell+3\leq k\leq 3\ell+2$, we establish a similar result for $q$-tight $(k,\ell)$-stable graphs, albeit with a worse upper bound. However, the condition $k\geq 2\ell+3$ is best possible when $q=0$, as every odd cycle on at least $2\ell+3$ vertices is tight $(2\ell+2,\ell)$-stable (see Lemma \ref{lem:odd-cycle}).

\begin{restatable}{theorem}{qtightwide}\label{thm:2l+3_bound}
    For nonnegative integers $k, \ell, q$ with   $k\geq 2\ell+3 $, if $G$ is  a $q$-tight $(k,\ell)$-stable graph on $n$ vertices, then     
    $$n\leq (q+1)(\ell+1) (k+2q-2\ell-1)^2 2^{6(k-2\ell+2q)^2}.$$
\end{restatable}


The paper is organized as follows. We prove Theorem \ref{thm:q-tight(k,l)} and Theorems~\ref{thm:q-tight(k,0)-stable} in Section~\ref{sec:main2}, and  Theorem \ref{thm:2l+3_bound} in Section~\ref{sec:2l+3}. We end the paper with some concluding remarks.

\section{Preliminaries}
Following Wu and Luo~\cite{DL24}, our proofs  also utilize properties of $\alpha$-critical graphs, where 
a graph $G$ is  \textit{$\alpha$-critical} if $\alpha(G-e)>\alpha(G)$ for every edge $e\in E(G)$. We define  $d(G):= |V(G)|-2\alpha(G)$ to be  the \textit{deficiency} of $G$.  Theorem~\ref{thm:Lov-nomin} and Theorem~\ref{thm:Lov} are results of Lov\'asz~\cite{Lov78, Lov19} to bound the number of vertices of degree at least three in connected,  $\alpha$-critical graphs.

\begin{theorem}[\cite{Lov78}]\label{thm:Lov-nomin}
Let $G$ be a connected,  $\alpha$-critical graph with   deficiency  $d \geq 0$. Then $G$ has at most $2^{2^{d+2}}$ vertices of degree at least three.
 \end{theorem}

\begin{theorem}[\cite{Lov19}]\label{thm:Lov}
Let $G$ be a connected,  $\alpha$-critical graph with minimum degree at least $3$ and deficiency {$d \geq 2$}. Then $G$ has at most $2^{3 d^2}$ vertices.
\end{theorem}

We next make some observations on the deficiency of  $(k,\ell)$-stable graphs and $q$-tight $(k,\ell)$-stable graphs. Observation~\ref{obs0} is due to Dong and Luo~\cite[Theorem 1.3]{DL24}.

\begin{observation}[\cite{DL24}]\label{obs0}
Every $(3,0)$-stable graph has minimum degree at least three.
\end{observation}

\begin{observation}\label{obs1}
Every $(k, \ell)$-stable  graph with $k\ge 2\ell+3$ has deficiency at least two. In particular, every $(3,0)$-stable graph has deficiency at least two.   
\end{observation}
\begin{proof} Let $G$ be a $(k, \ell)$-stable  graph on $n$ vertices with $k\ge 2\ell+3$. By Theorem \ref{thm:independence_num_Dong_Wu}, 
\[\alpha(G)\le \floor*{\frac{n-k+1}{2}}+\ell \le \frac{n-k+1}{2}+\ell,\]
    and so $ d(G)= n-2\alpha(G)\ge k-2\ell-1\ge  (2\ell+3)-2\ell-1 = 2$, as desired.    
\end{proof}

 \begin{observation}\label{obs2}
For all integers $k, \ell, q$ with $k>\ell$ and $q\ge0$, every $q$-tight $(k, \ell)$-stable  graph $G$  satisfies  $k-2\ell+2q -1\le d(G)\le k-2\ell+2q$.   
\end{observation}
\begin{proof} Let $G$ be a $q$-tight $(k, \ell)$-stable  graph on $n$ vertices. Then  
\[ \frac{n-k}{2}+\ell-q\le \alpha(G)= \floor*{\frac{n-k+1}{2}}+\ell-q\le \frac{n-k+1}{2}+\ell-q.\]
   It follows that  $ k-2\ell+2q-1\le d(G)= n-2\alpha(G)\le k-2\ell+2q$, as desired.    
\end{proof}

\section{Proofs of Theorem~\ref{thm:q-tight(k,l)} and  Theorem~\ref{thm:q-tight(k,0)-stable}}\label{sec:main2}

We begin this section with a lemma on the number of vertices of $q$-tight $(3,0)$-stable graphs. Note that for all $k\ge3$,  the graph obtained from a  $q$-tight $(k,0)$-stable graph   by deleting any $k-3$ vertices is  $q$-tight $(3,0)$-stable. Combining this with Lemma~\ref{lem: q-tight (3,0)-stable} leads to Theorem~\ref{thm:q-tight(k,0)-stable}. 
 \begin{lemma} \label{lem: q-tight (3,0)-stable}
 For all $q\ge 0$, every $q$-tight $(3,0)$-stable graph has at most $2^{3(3+2q)^2}$ vertices. 
\end{lemma}

\begin{proof}
We proceed the proof by  induction on $q$. By Theorem \ref{thm:(k,0)-tight-vertices}, every tight $(3,0)$-stable graph has at most $9<2 ^{27}$ vertices, and so the statement holds when 
 $q=0$. We may   assume that $q\ge 1$, and the statement holds for 
 any $q'$-tight $(3,0)$-stable graph, where $0\le q'<q$.
Let $G$ be a $q$-tight $(3,0)$-stable graph on $n$ vertices. Then 
\[ \alpha(G)=\floor*{\frac{n-3+1}{2}}-q=\floor*{\frac{n-2}{2}}-q.\]
Let $G'$ be an $\alpha$-critical spanning subgraph of $G$ such that $\alpha(G')=\alpha(G)$. Note that $G'$ is also $q$-tight $(3,0)$-stable. By Observation~\ref{obs2}, $ 2q +2\le d(G')\le  2q+3$. 
Suppose $G'$ is connected.   By Theorem \ref{thm:Lov}, $G'$ has at most $2^{3 (3+2q)^2}$ vertices.
 We may  assume that  $G'$ is disconnected. Let $C_1, C_2, \ldots, C_t$ be the components of $G'$, where $t\ge 2$, and let $n_i:=|V(C_i)|$ for each $i\in [t]$. Observe that $C_i$ is $(3,0)$-stable. Then $C_i$ is $q_i$-tight $(3,0)$-stable for some integer $q_i\ge 0$. This gives 

\[\frac{n-3}{2}-q \le \floor*{\frac{n-2}{2}}-q =\alpha(G)=\alpha(G')=\sum_{i=1}^t \alpha(C_i)=\sum_{i=1}^t \left(\floor*{\frac{n_i-3+1}{2}}-q_i\right)\le \frac{n}{2}-t - \sum_{i=1}^t q_i.\]
It follows that  $\sum_{i=1}^t q_i \le q + 3/2 -t$, and so  $\sum_{i=1}^t q_i \le q + 1 -t$ as $q_i, q, t$ are all integers. Hence $t\le q + 1-\sum_{i=1}^t q_i \le q+1$,  $\sum_{i=1}^t q_i \le q+1-2=q-1$ as $t\ge 2$, and  so $0\le q_i \le q-1$ for each $i\in [t]$.  We may assume that $G'$ has exactly $t_0$ components that are tight $(3,0)$-stable. Then $0\le t_0\le t$. By the induction hypothesis, 
$n_i\le 2 ^{3(3+2q_i)^2}$ for every $i\in [t]$, moreover  $n_i\le 9$  if $q_i=0$.   This implies that 
\begin{align*}
n=\sum_{i=1}^t n_i &=\sum_{i\in[t]: q_i=0} n_i +\sum_{i\in [t]: q_i\ge 1} n_i  
\\& \le \sum_{i\in[t]: q_i=0} 9 +\sum_{i\in [t]: q_i\ge 1} 2 ^{3(3+2q_i)^2} \\
&=9t_0+\sum_{i\in [t]: q_i\ge 1} 2 ^{3(3+2q_i)^2}.
\end{align*}
Note that the function $f(x)=2^{3(3+2x)^2}$ is increasing and satisfies that $f(x)+f(y)\le f(x+y)$ for $x,y\ge 1$. Hence
\[\sum_{i\in [t]: q_i\ge 1} 2 ^{3(3+2q_i)^2}= \sum_{i\in [t]: q_i\ge 1} f(q_i) \le f(\sum_{i\in [t]: q_i\ge 1} q_i)\le f(\sum_{i=1}^t q_i) \le f(q+1-t).
\]
This gives \[n \le 9t_0+\sum_{i\in [t]: q_i\ge 1} 2 ^{3(3+2q_i)^2} \le 9t+f(q+1-t)=9t+2^{3(3+2q+2-2t)^2} \le 2^{3(3+2q)^2},\]
 because  $2\le t \le q+1$ and the function $9t+2^{3(3+2q+2-2t)^2}$ is decreasing in $t$ over the interval  $ [2,q+1]$. Therefore, every $q$-tight $(3,0)$-stable graph has at most $2^{3(3+2q)^2}$ vertices.
\end{proof}
 We are now ready to prove Theorem~\ref{thm:q-tight(k,l)}, which we restate it here for convenience.\qtight*
\begin{proof}
We apply induction on $\ell$.  By Theorem \ref{thm:q-tight(k,0)-stable}, the statement holds for $\ell=0$.   We may assume that $\ell\ge 1$ and  the statement holds for  every  $q$-tight $(k',\ell')$-stable graph with $k'\ge 3\ell'+3$, where $0\le \ell'<\ell$.
Let  $G$ be a $q$-tight $(k, \ell)$-stable graph on $n$ vertices with $k\ge 3\ell+3$. Then 
\[\alpha(G) =\floor*{\frac{n-k+1}{2}}+\ell-q.\] 
We choose $S\subseteq V(G)$ such that $\alpha(G-S)=\alpha(G)-1$ and $|S|$ is minimum. Let $s:=|S|$.  Observe that $s\ge 1$; by the choice of $S$,   $G$ is $(s-1, 0)$-stable  when $s\ge 2$. Suppose $s\ge k-3\ell+1 $. 
Then $G$ is $(k-3\ell, 0)$-stable because $k\ge 3\ell+ 3$ and $G$ is $(k, \ell)$-stable. Let 
\[p:=\floor*{\frac{n-(k-3\ell)+1}{2}}-\alpha(G)=\floor*{\frac{n-(k-3\ell)+1}{2}}-\floor*{\frac{n-k+1}{2}}-\ell+q.\]
Then $p\le(\ell+1)/2+q$ and  $G$ is $p$-tight $(k-3\ell, 0)$-stable. By Theorem~\ref{thm:q-tight(k,0)-stable},   
\[n\le (k-3\ell-3)+2^{3(3+2p)^2}\le k-3\ell-3+ 2^{3(3+\ell+1+2q)^2}= k-3\ell-3 +2^{3(4+\ell+2q)^2},\] 
as desired. We may assume that  $s\le k-3\ell$. Then $k-s\ge 3\ell =3(\ell-1)+3>\ell-1$. Note that $G-S$ has $n-s$ vertices and $\alpha(G-S)=\alpha(G)-1$. Since $G$ is $(k,\ell)$-stable, we see that $G-S$ is $(k-s, \ell-1)$-stable. Observe that 
\[\alpha(G-S)=\alpha(G)-1=\floor*{\frac{n-k+1}{2}}+\ell-q-1=\floor*{\frac{n-s-(k-s)+1}{2}}+(\ell-1)-q,\]
and it follows that $G-S$ is $q$-tight $(k-s,\ell-1)$-stable. 
Recall that $k-s \ge 3(\ell-1)+3$.  By the induction hypothesis,  $G-S$ has at most $(k-s-3(\ell-1)-3)+2^{3(4+(\ell-1)+2q)^2}$ vertices. Therefore,  
 \[n=|S|+|V(G-S)|\le s+(k-s-3\ell)+2^{3(3+\ell+2q)^2}\le k-3\ell-3 + 2^{3(4+\ell+2q)^2}.\] 
 
This completes the proof of Theorem~\ref{thm:q-tight(k,l)}.
\end{proof}

\section{Proof of  Theorem~\ref{thm:2l+3_bound}}\label{sec:2l+3}

As mentioned earlier, it would be interesting to know whether  Theorem \ref{thm:q-tight(k,l)} holds when $k\le  3\ell+2$.  We begin this section with a lemma which demonstrates that $k$ needs to be larger than $2\ell+2$.
\begin{lemma}\label{lem:odd-cycle}
 For all $\ell\ge 0$, every odd cycle on at least $2\ell+3$ vertices is tight $(2\ell+2, \ell)$-stable.   
\end{lemma}
\begin{proof}
Let $C$ be an odd cycle on $n\ge 2\ell+3$ vertices.   Let $S\subseteq V(C)$ such that $|S|=2\ell+2$. Then $C-S$ is a disjoint union of paths. Denote $C-S$ by $\cup_{i=1}^t P_i$, where $P_i$ is a path on $n_i$ vertices in $C$ for each $i\in [t]$.   Since $\alpha(P_i)=\ceil*{n_i/2}\ge n_i/2$, we see that 
\[\alpha(C-S)=\sum_{i=1}^t\alpha(P_i)\ge \sum_{i=1}^t \frac{n_i}{2}=\frac{n-|S|}{2}=\frac{n-1}{2}-\ell-\frac{1}{2}=\alpha(C)-\ell-\frac{1}{2}.\]
This gives $\alpha(C-S)\ge \alpha(C)-\ell$ as $\alpha(C-S), \alpha(C), \ell$ are all integers. Hence, $C$ is $(2\ell+2, \ell)$-stable. Note that 
\[\alpha(C)=\frac{n-1}{2}=\floor*{\frac{n-(2\ell+2)+1}{2}}+\ell.\] Thus, $C$ is tight $(2\ell+2, \ell)$-stable, as desired.  
\end{proof}

Before we prove Theorem \ref{thm:2l+3_bound}, we  need to introduce more properties of    $\alpha$-critical graphs.  Erd\H{o}s and Gallai \cite{Erdos-Gallai1961} proved  that the deficiency of an $\alpha$-critical graph is always non-negative, and the only connected, $\alpha$-critical graph with deficiency $0$ is $K_2$. Theorem \ref{thm:Hajnal-max-degree} is a result of Hajnal~\cite{Hajnal1965} from 1965 on the maximum degree of an $\alpha$-critical graph. 

\begin{theorem}[\cite{Hajnal1965}]\label{thm:Hajnal-max-degree}
    Let $G$ be an $\alpha$-critical graph with deficiency $d\geq 0$. Then $\Delta(G)\leq d+1$.
\end{theorem}

 Theorem \ref{thm:Hajnal-max-degree} implies that the only connected, $\alpha$-critical graphs with deficiency $1$ are the odd cycles. The structures of connected, $\alpha$-critical graphs for deficiency $2$ and $3$ are classified by Andr\'asfai~\cite{And67} and Sur\'anyi~\cite{Sur73},  respectively.  We next introduce an operation on $\alpha$-critical graphs due to Lov\'asz~\cite{Lov78, Lov19}.

A \textit{split} operation on a vertex $v$ in a graph $G$ with $d_G(v)\ge2$ is to first  replace $v$ with two new non-adjacent vertices $v'$ and $v''$ such that $N_G(v')\cap N_G(v'')=\emptyset$, $\min\{|N_G(v')|, |N_G(v'')|\}\geq 1$ and $N_G(v')\cup N_G(v'') = N_G(v)$, and  then connect both $v'$ and $v''$ to a new vertex $w$. Observe that given an edge $uv$ (where $d_G(v)\geq 2$), splitting the vertex $v$ into $v'$, $v''$ such that $N_G(v') = \{u\}$ is equivalent to subdividing the edge $uv$ by two vertices (which are $v'$ and the new vertex $w$).  Lov\'asz~\cite{Lov78, Lov19} made the following two observations.

\begin{observation}[\cite{Lov78, Lov19}]\label{obs:split} For any graph $G$, 
    \begin{enumerate}[(i)]
        \item if $G'$ is obtained from $G$ by splitting a vertex $v\in V(G)$, then $G'$ is $\alpha$-critical if and only if $G$ is, and $d(G)= d(G')$. 
        \item if $G\ne K_3$ is an $\alpha$-critical graph and $u\in V(G)$ is a vertex of degree two with neighbors $v, w$. Then  $vw\notin E(G)$,  and  $u$ is the only common neighbor of $v$ and $w$.
    \end{enumerate}
\end{observation}

We also need the following observation on $\alpha$-critical graphs, which we prove it here for completeness.
\begin{observation}\label{obs:critical_graph_min_degree}
   Let $G$ be a connected $\alpha$-critical graph with $|V(G)|\geq 3$. Then $\delta(G)\geq 2$. 
\end{observation}
\begin{proof}
It suffices to prove that $G$ has no cut vertex. Suppose otherwise that $G$ has a cut vertex $v$ and the components $G-v$ are $C_1, C_2, \ldots, C_t$ with $t\geq 2$. Observe that since $G$ is $\alpha$-critical, every vertex in $G$ belongs to a maximum independent set of $G$. Let $S$ be a maximum independent set in $G$ containing $v$. 
For each $i\in [t]$, let $x_i$ be a neighbor of $v$ in $C_i$. Since $G$ is $\alpha$-critical, $\alpha(G-vx_i)= \alpha(G)+1$, which implies that there exists an independent set of size $\alpha(G)+1$ in $G-vx_i$ containing both $v$ and $x_i$. It follows that $\alpha(C_i)\geq |S\cap V(C_i)|+1$ for every $i\in [t]$. Thus
$$\alpha(G)\geq \sum_{i\in [t]} \alpha(C_i) \geq \sum_{i\in [t]} (|S\cap V(C_i)|+1) >|S|,$$
since $t\geq 2$, contradicting that $S$ is a maximum independent set in $G$.
\end{proof}

Using Theorem \ref{thm:Lov} and Observation \ref{obs:split}, we are able to extend Theorem~\ref{thm:Lov-nomin} further. 

\begin{lemma}\label{lem:num_of_degree_3_vertex}
    Let $G$ be a connected,  $\alpha$-critical graph with deficiency {$d \geq 2$}. Then $G$ has at most $(d-1)\cdot 2^{3d^2}$ vertices of degree at least three.
\end{lemma}
\begin{proof} Let $G$ be a connected,  $\alpha$-critical graph with deficiency {$d \geq 2$}. Then $|V(G)|\ge 4$. By Observation \ref{obs:critical_graph_min_degree}, $\delta(G)\ge 2$. 
    By Observation \ref{obs:split} and Theorem \ref{thm:Lov}, we see that every connected, $\alpha$-critical graph with deficiency $d\geq 2$ can be obtained from some connected, $\alpha$-critical  graph $H$ with   $|V(H)|\le 2^{3d^2}$   and   $\delta(H)\ge3$ by iteratively splitting a vertex. Observe that splitting a vertex $v$ in $H$ increases the number of vertices of degree at least three in $H$ by one if and only if $d_H(v)\geq 4$ and the two vertices $v',v''$ replacing $v$ each is adjacent to at least two of the neighbors of $v$ in $H$. It follows that for every vertex $v\in V(H)$ with $d_H(v)\ge4$,  the number of vertices of degree at least three obtained from iteratively splitting $v$ and its replacement vertices (if $v$ is split) is at most $d_H(v)-2$. 
    By Theorem \ref{thm:Hajnal-max-degree}, $\Delta(H)\leq d+1$.
    Therefore the number of vertices of degree at least three in $G$ is at most  $$|V(H)| \cdot (\Delta(H) -2)\le |V(H)| \cdot (d+1-2)\leq (d-1)\cdot 2^{3d^2}.\qedhere$$

\end{proof}

\begin{lemma}\label{lem:edge-bound-critical}
   Let $G$ be an $\alpha$-critical graph with $k\geq 2\ell+3 $ and deficiency $d\geq 2$. If $G$ is $(k,\ell)$-stable, then $G$ has at most $2\ell\binom{(d-1)\cdot 2^{3d^2}}{2}$ vertices of degree two.
\end{lemma}
\begin{proof} 
    Let $G$ be a connected, $\alpha$-critical graph with $k\geq 2\ell+3\geq 3$ and deficiency $d\geq 2$. Again by Observation \ref{obs:critical_graph_min_degree}, we have $\delta(G)\ge2$. 
    By Lemma \ref{lem:num_of_degree_3_vertex}, $G$ has at most $(d-1)\cdot 2^{3d^2}$ vertices of degree at least $3$.  Recall that $G$ can be obtained from some connected, $\alpha$-critical  graph   $H_0$ with $|V(H_0)|\le 2^{3d^2}$   and   $\delta(H_0)\ge3$ by iteratively splitting a vertex, and each graph in the process is $\alpha$-critical.

    We claim that for any two vertices $x,y$ of degree at least three in $G$, there exists at most one path between $x$ and $y$ in $G$ such that each internal vertex (if it exists) in the path has degree exactly two. Note that $H_0$ satisfies the claim since  $\delta(H_0)\ge3$. By Observation \ref{obs:split} (ii), if an $\alpha$-critical graph  satisfies the claim, then the graph obtained from splitting a vertex in it also satisfies the claim (as they are both $\alpha$-critical). Thus the claim holds for $G$. It then follows that 
  $G$ can be obtained from some connected, $\alpha$-critical  graph $H$ with   $|V(H)|\le (d-1)\cdot 2^{3d^2}$   and   $\delta(H)\ge3$  by iteratively subdividing the edges of $H$. Then  $|E(H)|\leq \binom{(d-1) \cdot 2^{3d^2}}{2}$. 

Suppose for contradiction that $G$ is $(k,\ell)$-stable and has more than  $2\ell\binom{(d-1)\cdot 2^{3d^2}}{2}$ vertices of degree two.   By the Pigeonhole Principle,  some edge of  $H$ must be subdivided at least $(2\ell+1)$ times. That is, there exists an induced path $P$ in $G$ with vertices $v_1,  v_2, \ldots,  v_{2\ell+3}$  in order  such that $d_G(v_i)=2$ for each $i$ satisfying $2\leq i\leq 2\ell+2$. Since $G$ is $(k,\ell)$-stable with $k\geq 2\ell+3$, we see that 
    $$\alpha(G-V(P)) \geq \alpha(G)-\ell.$$
    But then  there exists an independent set $S$ of size   $\alpha(G)-\ell$ in $G-V(P)$; and $S\cup \{v_2, v_4, \cdots ,v_{2\ell+2}\}$ is  an independent set of size   $\alpha(G)+1$ in $G$, a contradiction.\end{proof}

We are now ready to prove Theorem \ref{thm:2l+3_bound}, which we restate it here for convenience.
\qtightwide*

\begin{proof}
Let $G$, $k$, $\ell$, $q$ and $n$  be  as given in the statement. 
Let $G'$ be an $\alpha$-critical spanning subgraph of $G$ such that $\alpha(G') = \alpha(G)$. By Observation \ref{obs:critical_graph_min_degree}, $\delta(G)\geq 2$. Observe that $G'$ is also $q$-tight $(k,\ell)$-stable. 
By Observation~\ref{obs2} and the fact that $k\ge 2\ell+3$ and $q\ge0$,    
\[2\le 2q+2 \le k-2\ell+2q -1\le d(G')\le k-2\ell+2q.\]
Let $C_1, \ldots, C_t$ be the components of $G'$, where $t\ge1$ is an integer. For each $i\in [t]$, let $n_i:=|V(C_i)|$.
Observe that for each $i\in [t]$, $C_i$ is  $(k,\ell)$-stable as $G'$ is  $(k,\ell)$-stable. For each $i\in [t]$, let $A_i$ be the set of vertices of degree two in $C_i$. 
    If $t=1$, then $|V(G')-A_1|\le (d(G')-1)\cdot 2^{3d(G')^2}$ by Lemma \ref{lem:num_of_degree_3_vertex} and $|A_1| \leq 2\ell\binom{(d(G')-1)\cdot 2^{3d(G')^2}}{2}$ by Lemma \ref{lem:edge-bound-critical}. It follows that 
    \begin{align}
        n=|A_1|+|V(G')-A_1| & \le 2\ell\binom{(d(G')-1)\cdot 2^{3d(G')^2}}{2}+ (d(G')-1)\cdot 2^{3d(G')^2} \notag\\
                        & \le (\ell+1) (d(G')-1)^2 2^{6d(G')^2}\notag\\
                        &\leq (\ell+1) (k+2q-2\ell-1)^2 2^{6(k-2\ell+2q)^2} \tag{$*$}\label{eq:vtx_bound_1.9},
    \end{align}
    as desired. We may assume that $t\geq 2$. Observe that 

    $$\floor*{\frac{n-k+1}{2}}+\ell-q= \alpha(G') = \dss_{i=1}^t \alpha(C_i) \leq \dss_{i=1}^t \lp\floor*{\frac{n_i-k+1}{2}}+\ell \rp$$

    We claim that $C_i$ is  $q_i$-tight $(k,\ell)$-stable for each $i\in [t]$, where $q_i\le q$.
    Otherwise, without loss of generality, suppose $C_1$ is not $q_1$-tight $(k,\ell)$-stable for each $q_1\le q$. Then $\alpha(C_1)  \leq \frac{n_1-k+1}{2}+\ell-(q+1)$ and so 
    \begin{align*}\alpha(G')=\dss_{i=1}^t \alpha(C_i) & \leq \frac{n_1-k+1}{2}+\ell-(q+1) + \dss_{i=2}^t \lp\frac{n_i-k+1}{2}+\ell \rp\\
    & \leq \frac{n-k+1}{2} + \ell-(q+1) - \frac{(t-1)(k-1-2\ell)}{2}\\
    &\leq \frac{n-k+1}{2} +\ell-q -1-  \frac{(t-1)(2\ell+3-1-2\ell)}{2} \\
       &\leq \frac{n-k+1}{2} +\ell-q -1- (t-1)\\
    & < \floor*{\frac{n-k+1}{2}}+\ell-q\\
    &=\alpha(G'),
    \end{align*}
which is impossible. Therefore, $C_i$ is  $q_i$-tight $(k,\ell)$-stable for each $i\in [t]$ where $q_i\le q$. By Observation~\ref{obs2},  
 \[2\le d(C_i)  \le k-2\ell+2q_i\le k-2\ell+2q.\]
    Hence for each $i\in[t]$, similar to \eqref{eq:vtx_bound_1.9}, by Lemma~\ref{lem:num_of_degree_3_vertex} and Lemma~\ref{lem:edge-bound-critical}, 
    $$|V(C_i)|\leq |A_i|+|V(C_i)-A_i|\leq (\ell+1) (k+2q-2\ell-1)^2 2^{6(k-2\ell+2q)^2}.$$
Since $G'$ is $q$-tight $(k,\ell)$-stable, we have
    \begin{align*}\frac{n-k}{2}+\ell-q 
    \leq \floor*{\frac{n-k+1}{2}}+\ell-q 
    = \alpha(G') &=\dss_{i=1}^t \alpha(C_i) \\
    & \leq \dss_{i=1}^t \lp\floor*{\frac{n_i-k+1}{2}}+\ell \rp \\
    &\leq \dss_{i=1}^t \lp\frac{n_i-k+1}{2}+\ell \rp\\
    &\leq \frac{n-k}{2}- \frac{(t-1)(k-1)}{2}+t\ell+\frac12.
    \end{align*}
This, together with the fact that $k\geq 2\ell+3$, implies that
$$t\leq 1+\frac{2q+1}{k-2\ell-1} \leq 1+ \frac{2q+1}{2}\leq q+\frac{3}{2},$$
 as $k\geq 2\ell+3$. Since $t$ is an integer, we then have $t\leq q+1$. It follows that
\begin{align*}n&=\dss_{i=1}^t  |V(C_i)|\\
&\leq (q+1)(\ell+1) (k+2q-2\ell-1)^2 2^{6(k-2\ell+2q)^2}.
\end{align*}
This completes the proof of Theorem~\ref{thm:2l+3_bound}.  
\end{proof}

\section{Concluding Remarks}

In our proofs of Theorem \ref{thm:q-tight(k,l)} and Theorem \ref{thm:2l+3_bound}, we heavily use Theorem \ref{thm:Lov}, which bound the number of vertices in a connected, $\alpha$-critical graph with deficiency $d$ and minimum degree at least three. Lov\'asz remarked in \cite{Lov19} that the upper bounds in Theorem \ref{thm:Lov-nomin} and Theorem \ref{thm:Lov} are probably very rough. An example of Sur\'anyi~\cite{Sur73} showed that there exist $\alpha$-critical graphs with deficiency $d$ and $\Omega(d^2)$   vertices with degree at least three. It would be interesting to find a better upper bound for Theorem \ref{thm:Lov}.

In another direction, Dong and Wu~\cite{DW22} constructed a sequence of $n$-vertex $(3,0)$-stable graphs with independence number $n/2-O(\sqrt{n})$. Alon \cite{Alo21} extended the above result by showing that for every $k>\ell\geq 0$, there exists a sequence of $n$-vertex $(k,\ell)$-stable graphs with independence number $n/2-o(n)$.    Theorem \ref{thm:q-tight(k,l)} states that for $k\geq 3\ell+3\geq 3$, there exists no constant $c$ such that there are infinite $(k,\ell)$-stable graphs $G$ with $\alpha(G)\geq |V(G)|/2-c$. For nonnegative integers $k, \ell$ with  $k\geq 3\ell+3$, let  $f(n)$ be a function of $n$ such that there exists a sequence of $n$-vertex $(k,\ell)$-stable graphs $G$ with $n\to\infty$ and $\alpha(G)\geq n/2-f(n)$. It would be interesting to know whether $f(n) = \Omega(\sqrt{n})$.

\end{document}